\documentclass[12pt]{amsart}
 \usepackage{amscd,amsmath,amsthm,amssymb}
 \usepackage{pstcol,pst-plot,pst-3d}%\usepackage[T1]{fontenc}
 \psset{unit=0.7cm,linewidth=0.8pt,arrowsize=2.5pt 4}
\newpsstyle{fatline}{linewidth=1.5pt}
\newpsstyle{fyp}{fillstyle=solid,fillcolor=verylight}
\definecolor{verylight}{gray}{0.97}
\definecolor{light}{gray}{0.93}
\definecolor{medium}{gray}{0.82}
 \def\frk{\frak}               % font for "Fraktur"

 \def\mm{{\frk m}}
 
 \def\nn{{\frk n}}
 \def\opn#1#2{\def#1{\operatorname{#2}}} % to make operators
 \opn\chara{char} \opn\length{\ell} \opn\pd{pd} \opn\rk{rk}
 \opn\projdim{proj\,dim} \opn\injdim{inj\,dim} \opn\rank{rank}
 \opn\depth{depth} \opn\grade{grade} \opn\height{height}
 \opn\embdim{emb\,dim} \opn\codim{codim}
 
 \opn\Tr{Tr} \opn\bigrank{big\,rank}
 \opn\superheight{superheight}\opn\lcm{lcm}
 \opn\trdeg{tr\,deg}%\emph{
 \opn\reg{reg} \opn\lreg{lreg} \opn\ini{in} \opn\lpd{lpd}
 \opn\size{size} \opn\sdepth{sdepth}
 \opn\link{link}\opn\fdepth{fdepth}\opn\lex{lex}
 \opn\div{div} \opn\Div{Div} \opn\cl{cl} \opn\Cl{Cl}
 \opn\Spec{Spec} \opn\Supp{Supp} \opn\supp{supp} \opn\Sing{Sing}
 \opn\Ass{Ass} \opn\Min{Min}\opn\Mon{Mon}
 \opn\Ann{Ann} \opn\Rad{Rad} \opn\Soc{Soc}
 \opn\Im{Im} \opn\Ker{Ker} \opn\Coker{Coker} \opn\Am{Am}
 \opn\Hom{Hom} \opn\Tor{Tor} \opn\Ext{Ext} \opn\End{End}
 \opn\Aut{Aut} \opn\id{id}
 
 \opn\nat{nat}
 \opn\pff{pf}%   \pf exists already
 \opn\Pf{Pf} \opn\GL{GL} \opn\SL{SL} \opn\mod{mod} \opn\ord{ord}
 \opn\Gin{Gin} \opn\Hilb{Hilb}\opn\sort{sort}
 \opn\aff{aff} \opn
 \con{conv} \opn\relint{relint} \opn\st{st}
 \opn\lk{lk} \opn\cn{cn} \opn\core{core} \opn\vol{vol}
 \opn\link{link} \opn\star{star}\opn\lex{lex}\opn\set{set}
 \opn\gr{gr}
 
 \def\pot#1#2{#1[\kern-0.28ex[#2]\kern-0.28ex]}

 \opn\dirlim{\underrightarrow{\lim}}
 \opn\inivlim{\underleftarrow{\lim}}
 \let\union=\cup

 \let\Union=\bigcup

 \def\Implies{\ifmmode\Longrightarrow \else
         \unskip${}\Longrightarrow{}$\ignorespaces\fi}
 \def\implies{\ifmmode\Rightarrow \else
         \unskip${}\Rightarrow{}$\ignorespaces\fi}
 \def\iff{\ifmmode\Longleftrightarrow \else
         \unskip${}\Longleftrightarrow{}$\ignorespaces\fi}

 \let\:=\colon
 \newtheorem{Theorem}{Theorem}[section]

 \let\epsilon\varepsilon
 \let\kappa=\varkappa
 \def\qed{\ifhmode\textqed\fi
       \ifmmode\ifinner\quad\qedsymbol\else\dispqed\fi\fi}
 \def\textqed{\unskip\nobreak\penalty50
        \hskip2em\hbox{}\nobreak\hfil\qedsymbol
        \parfillskip=0pt \finalhyphendemerits=0}
 \def\dispqed{\rlap{\qquad\qedsymbol}}
 
 \opn\dis{dis}
 \def\pnt{{\raise0.5mm\hbox{\large\bf.}}}
 
 \opn\Lex{Lex}

\begin{document}
% \linenumbers

 \title {Monomial ideals whose depth function has any given number of strict local maxima}

 \author {Somayeh Bandari, J\"urgen Herzog and Takayuki Hibi}

\address{S. Bandari, Department of Mathematics, Az-Zahra
University, Vanak, Post Code 19834, Tehran, Iran.}
\email{somayeh.bandari@yahoo.com}

\address{J\"urgen Herzog, Fachbereich Mathematik, Universit\"at Duisburg-Essen, Campus Essen, 45117
Essen, Germany} \email{juergen.herzog@uni-essen.de}

\address{Takayuki Hibi, Department of Pure and Applied Mathematics, Graduate School of Information Science and Technology,
Osaka University, Toyonaka, Osaka 560-0043, Japan}
\email{hibi@math.sci.osaka-u.ac.jp}

 \begin{abstract}
We construct  monomial ideals with the property that their depth function has any given number of strict local maxima.
 \end{abstract}

 \maketitle

\section*{}
In recent years there have been several publications concerning the stable set of prime ideals of a monomial ideal, see for example  \cite{CMS},\cite{FHV}, \cite{MMV} and \cite{HRV}.  It is known by Brodmann \cite{Br2} that for any graded ideal $I$ in the polynomial ring $S$ (or any proper ideal $I$ in a local ring)  there exists an integer $k_0$ such that  $\Ass(I^k)=\Ass(I^{k+1})$ for $k\geq k_0$. The smallest integer $k_0$ with this property is called the index of stability of $I$ and $\Ass(I^{k_0})$ is  called the set of stable prime ideals of $I$. A  prime ideal $P\in \Union_{k\geq 1}\Ass(I^k)$ is said to be persistent  with respect to $I$ if whenever $P\in \Ass(I^k)$ then $P\in \Ass(I^{k+1})$, and the ideal $I$ is said to satisfy the persistence property  if all prime ideals $P\in \Union_{k\geq 1}\Ass(I^k)$ are persistent. It is an open question (see \cite{FHT} and \cite[Question 3.28]{VT}) whether  any squarefree monomial ideal satisfies the persistence property.

We call the numerical function $f(k)=\depth(S/I^k)$ the depth function of $I$. It is easy to see that a monomial ideal  $I$ satisfies the persistence property  if all monomial localizations of $I$ have a non-increasing depth function. In view of the above mentioned open question it is natural to ask whether all squarefree monomial ideals have non-increasing depth functions.  The situation for non-squarefree monomial ideals  is completely different. Indeed, in \cite[Theorem 4.1]{HH1} it is shown that for any non-decreasing numerical function $f$, which is eventually constant,  there exists a monomial ideal $I$ such that $f(k)=\depth(S/I^k)$
for all $k$.  Note that a similar result for non-increasing depth functions is not known, even it is expected that all squarefree monomial ideals have  non-increasing depth functions. In general  the depth function of a monomial ideal does not  need to be  monotone. Examples of monomial ideals with non-monotone depth function are given in \cite[Example 4.18]{MV} and \cite{HH1}. The question arises which numerical functions are   depth functions of  monomial ideals. Since $\depth(S/I^k)$  is constant for all $k\gg 0$ (see \cite{Br1}), any depth functions is eventually  constant. So the most wild conjecture one could make is that any numerical function  which is eventually constant is indeed the  depth function of a monomial ideal. In support of this conjecture we show in our theorem that for any given number $n$ there exists a monomial ideal whose depth function  has precisely $n$ strict local maxima. The price that we  have to pay to obtain such examples is that the number of variables needed to define  our monomial ideal with $n$ strict local maxima is relatively large, namely $2n+4$. For this class of examples  the depth function is constant  beyond the number of variables. In all other examples  known to us, in particular those discussed in \cite{HH1}, this is also the case.  Thus we are tempted to conjecture that for any monomial ideal $I$ in a polynomial ring in $n$ variables $\depth(I^k)$   is constant  for  $k\geq n$.

In the following theorem we present  the  monomial  ideals admitting a depth function as announced in the tile of the paper.

\begin{Theorem}
\label{strange}
Let $n\geq 0$ be an integer and $I\subset S=K[a,b,c,d, x_1,y_1,\ldots,x_n,y_n]$ be the monomial ideal in the polynomial ring $S$ with generators
\[
a^6,a^5b,ab^5,b^6,a^4b^4c,a^4b^4d,a^4x_1y_1^2,b^4x_1^2y_1,\ldots,a^4x_ny_n^2,b^4x_n^2y_n.
\]
Then
\[
\depth(S/I^k)=\left\{
              \begin{array}{ll}
                0, & \hbox{if $k$ is odd and $k\leq 2n+1$;} \\
                1, & \hbox{if $k$ is even and $k\leq 2n$;} \\
                2, & \hbox{if $k> 2n+1$.}
              \end{array}
            \right.
\]
In particular, the depth function of this  ideal has precisely $n$ strict local maxima.
\end{Theorem}

\begin{proof}
First of all, for each odd integer $k=2t-1$ with $t\leq n+1$,
we show that $\depth(S/I^k)=0$.  For this purpose we find a monomial belonging to
$(I^k\: \mm) \setminus I^{k}$,
where $\mm = (a,b,c,d, x_1,y_1,\ldots,x_n,y_n)$.
% Let $\mm = (a,b,c,d, x_1,y_1,\ldots,x_n,y_n)$.
We claim that the monomial
% $$u=a^{4t}b^{4t}(x_1y_1)^3\cdots (x_{t-1}y_{t-1})^{3}x_ty_t\cdots x_ny_n$$
$$u=a^{4}b^{4}(a^4x_1y_1^2)(b^4x_1^2y_1)\cdots(a^4x_{t-1}y_{t-1}^2)(b^4x_{t-1}^2y_{t-1})x_ty_t\cdots x_ny_n$$
satisfies $u \in (I^k\: \mm) \setminus I^{k}$.
% Let $$w=\prod_{i=2}^{t-2}(a^4x_iy_i^2)(b^4x_i^2y_i), \, \, w'=\prod_{i=2}^{t-1}(a^4x_iy_i^2)(b^4x_i^2y_i).$$
% Then
% \[
% u = a^{4}b^{4} (a^4x_1y_1^2)(b^4x_1^2y_1) w'x_ty_t\cdots x_ny_n
% \]
Let
\begin{eqnarray*}
v_1&=&a^5b\cdot b^6(a^4x_{t-1}y_{t-1}^{2})\prod_{i=1}^{t-2}(a^4x_iy_i^2)(b^4x_i^2y_i),\\
v_2&=&ab^5\cdot a^6(b^4x_{t-1}^2y_{t-1})\prod_{i=1}^{t-2}(a^4x_iy_i^2)(b^4x_i^2y_i),\\
v_3&=&a^4b^4c\prod_{i=1}^{t-1}(a^4x_iy_i^2)(b^4x_i^2y_i),\\
v_4&=&a^4b^4d\prod_{i=1}^{t-1}(a^4x_iy_i^2)(b^4x_i^2y_i),\\
\end{eqnarray*}
\begin{eqnarray*}
v_{2\ell+3}&=&a^6(b^{4}x_{\ell}^{2}y_{\ell})^{2}\prod_{i=1}^{\ell-1}(a^4x_iy_i^2)(b^4x_i^2y_i)
\prod_{i=\ell+1}^{t-1}(a^4x_iy_i^2)(b^4x_i^2y_i), \, \, \, 1\leq\ell\leq t-1,\\
v_{2\ell+4}&=&b^6(a^{4}x_{\ell}y_{\ell}^{2})^{2}\prod_{i=1}^{\ell-1}(a^4x_iy_i^2)(b^4x_i^2y_i)
\prod_{i=\ell+1}^{t-1}(a^4x_iy_i^2)(b^4x_i^2y_i), \, \, \, 1\leq\ell\leq t-1,\\
v_{2\ell+3}&=&(b^{4}x_{\ell}^{2}y_{\ell})\prod_{i=1}^{t-1}(a^4x_iy_i^2)(b^4x_i^2y_i),
\, \, \, t\leq\ell\leq n,\\
v_{2\ell+4}&=&(a^{4}x_{\ell}y_{\ell}^{2})\prod_{i=1}^{t-1}(a^4x_iy_i^2)(b^4x_i^2y_i),
\, \, \, t\leq\ell\leq n.\\
%
% v_{2\ell+1}&=&a^6(b^{4}x_{\ell-1}^{2}y_{\ell-1})^{2}\prod_{i=1}^{\ell-2}(a^4x_iy_i^2)(b^4x_i^2y_i)
% \prod_{i=\ell}^{t-1}(a^4x_iy_i^2)(b^4x_i^2y_i), \, \, \, 2\leq\ell\leq t,\\
% v_{2\ell+2}&=&b^6(a^{4}x_{\ell-1}y_{\ell-1}^{2})^{2}\prod_{i=1}^{\ell-2}(a^4x_iy_i^2)(b^4x_i^2y_i)
% \prod_{i=\ell}^{t-1}(a^4x_iy_i^2)(b^4x_i^2y_i), \, \, \, 2\leq\ell\leq t,\\
% v_{2\ell+1}&=&(b^{4}x_{\ell-1}^{2}y_{\ell-1})\prod_{i=1}^{t-1}(a^4x_iy_i^2)(b^4x_i^2y_i),
% \, \, \, t+1\leq\ell\leq n+1,\\
% v_{2\ell+2}&=&(a^{4}x_{\ell-1}y_{\ell-1}^{2})\prod_{i=1}^{t-1}(a^4x_iy_i^2)(b^4x_i^2y_i),
% \, \, \, t+1\leq\ell\leq n+1.\\
\end{eqnarray*}
Clearly $v_i\in I^k$ for $1 \leq i \leq 2n+4$.  One easily see that
\[
v_1\,|\,a\cdot u, \;\;\; v_2\,|\,b\cdot u, \;\;\; v_3\,|\,c\cdot u, \;\;\; v_4\,|\,d\cdot u.
\]
Since
\begin{eqnarray*}
a^{4}b^{4}(a^{4}x_{\ell}y_{\ell}^{2})x_{\ell}
= a^{2}(a^{6}(b^{4}x_{\ell}^{2}y_{\ell}))y_{\ell}, \,\,\,\,\,
a^{4}b^{4}(b^{4}x_{\ell}^{2}y_{\ell})y_{\ell}
= b^{2}(b^{6}(a^{4}x_{\ell}y_{\ell}^{2}))x_{\ell},
\end{eqnarray*}
it follows that
\[
v_{2\ell+3}\,|\,x_{\ell}\cdot u, \;\;\; v_{2\ell+4}\,|\,y_{\ell}\cdot u, \;\;\; 1\leq\ell\leq t-1.
\]
Moreover,
\[
v_{2\ell+3}\,|\,x_{\ell}\cdot u, \;\;\; v_{2\ell+4}\,|\,y_{\ell}\cdot u, \;\;\; t\leq\ell\leq n.
\]
Hence $u\cdot\mm \subseteq I^k$.  In other words, $u \in I^k\: \mm$.

Now, we wish to prove that $u\not\in I^k$.  Since neither $c$ nor $d$ divides $u$,
it is enough to show that $u\not\in {\bar{I}}^k$, where
\[
\bar{I}=(a^6,a^5b,ab^5,b^6,a^4x_1y_1^2,b^4x_1^2y_1,\ldots,a^4x_ny_n^2,b^4x_n^2y_n).
\]
Suppose that there exists a monomial $w=u_1\cdots u_k\in {\bar{I}}^k$ with each $u_i\in G(\bar{I})$
such that $w$ divides $u$.
Since $\deg_{x_i}(u)=\deg_{y_i}(u)=1$ for $i = t, \ldots, n$,
each $u_{i}$ belongs to
\def\M{{\mathcal M}}
\def\N{{\mathcal N}}
\[
\M = \{a^6,a^5b,ab^5,b^6,a^4x_1y_1^2,b^4x_1^2y_1,\ldots,a^4x_{t-1}y_{t-1}^2,b^4x_{t-1}^2y_{t-1}\}.
\]
Since  $\deg_{x_i}(u)=\deg_{y_i}(u)=3$ for $i=1,\ldots, t-1$, it follows that, for $u_i$ and $u_j$
belonging to
\[
\N = \{a^4x_1y_1^2,b^4x_1^2y_1,\ldots,a^4x_{t-1}y_{t-1}^2,b^4x_{t-1}^2y_{t-1}\}
\]
with $i\neq j$, one has $u_i\neq u_j$.
Since $|\N|=2t-2=k-1$, there exists $1 \leq j \leq k$ with $u_j\not\in \N$.
Let $\rho$ denote the number of integers $1 \leq j \leq k$ with $u_j\not\in \N$.
Since $w$ divides $u$, one has
$\deg_a(w)\leq 4t$ and $\deg_b(w)\leq 4t$.

\medskip
(a)  Let $\rho = 1$.
Since $|\N|=k-1$, each monomial belonging to $\N$ divides $w$.
Thus $\deg_a(w)=4(t-1)+c$  and $\deg_b(w)=4(t-1)+d$, where $(c,d)$ belongs to
$\{(0,6),(1,5),(5,1),(6,0)\}$.
Hence one has either $\deg_a(w)>4t$ or $\deg_b(w)>4t$, a contradiction.

\smallskip
(b)  Let $\rho = 2$.  Then we may assume that
$\deg_a(w)=4(t-1)+c_1+c_2$ and $\deg_b(w)=4(t-2)+d_1+d_2$,
where each $(c_i,d_i)$ belongs to $\{(0,6),(1,5),(5,1),(6,0)\}$.
Again, one has either $\deg_a(w)>4t$ or $\deg_b(w)>4t$, a contradiction.

\smallskip
(c) Let $\rho = h$ with $h>2$.
% $k-h=s+$
Suppose that $a^4$ divides each of the monomials $u_{1}, \ldots, u_{s}$, where $s \leq k - h$.
Let $\deg_a(w)=4s+c_1+\ldots+c_h$ and $\deg_b(w)=4(k-h-s)+d_1+\ldots+d_h$,
where $c_i+d_i=6$ for each $1 \leq i \leq h$.  Since
\[
\deg_b(w)=4(k-h-s)+(6 - c_{1})+\cdots+(6 - c_h) \leq 4t = 2(k+1),
\]
it follows that
\[
\deg_a(w)=4s+c_1+\ldots+c_h \geq 4(k-h) + 6h -2(k+1) = 2k + 2h -2.
\]
However, since $h > 2$, one has
\[
2k + 2h -2> 2k + 4 - 2 = 2k + 2 = 2(k+1) = 4t.
\]
Thus $\deg_a(w) > 4t$, a contradiction.

\medskip

The above discussions (a), (b) and (c) complete the proof of $u\not\in I^k$.
Hence $u$ belongs to $(I^k\: \mm) \setminus I^{k}$ and $\depth(S/I^k)=0$, as desired.

\medskip
Now we are going to prove  that $\depth(S/I^k) \geq 1$  for any even number $k>0$. For the  proof  we introduce the ideals  $J=(a^6,a^5b,ab^5,b^6,a^4b^4c,a^4b^4d)$ and
$L=(a^4x_1y_1^2,b^4x_1^2y_1,\ldots,a^4x_ny_n^2,b^4x_n^2y_n)$. Then $I=J+L$, and hence
\begin{eqnarray*}
\label{sum}
I^k=J^k+J^{k-1}L+\cdots +J^2L^{k-2}+JL^{k-1}+L^k.
\end{eqnarray*}

We first show that for $k\geq 2$,  the factor module  $(I^k\:(c,d))/I^k$ is generated  by the residue classes of the elements of set
\begin{eqnarray}
\label{cd}
\mathcal{S}_k= \{a^4b^4v_1\cdots v_{k-1}\:\; v_i\in G(L) \text{ and } v_i\neq v_j \text{ for } i\neq j\}.
\end{eqnarray}

Observe that the minimal set of generators of $J^2$ only consists of monomials in $a$ and $b$. Therefore, the only monomials in $I^k$ which are divisible by $c$ or  $d$ are the generators of $JL^{k-1}$.  It follows that the generators of $I^k\:(c,d)$ which do not belong to $I^k$ are the  monomials  of the form $a^4b^4v_1\cdots v_{k-1}$ with $v_i\in G(L)$.

Suppose that $v_i=v_j$ for some $i\neq j$, say, $v_i=v_j=a^4x_\ell y_\ell^2$. We may assume that $i=1$ and $j=2$. Then
\[
u=a^4b^4v_1\cdots v_{k-1}=a^{12}b^4x_\ell^2y_\ell^4v_3\cdots v_{k-1}=(a^{12})(b^4x_\ell^2y_\ell)v_3\cdots v_{k-1}y_\ell^3.
\]
Since $a^{12}\in J^2$ and since $b^4x_\ell^2y_\ell\in L$, we see that $u\in J^2L^{k-2}\subset I^k$. This proves (\ref{cd}).

\medskip
For  a monomial  $u=a^4b^4v_1\cdots v_{k-1}\in\mathcal{S}_k$, we set
\[
Z_u=\{x_\ell\:\; \deg_{x_\ell}(v_i)=1 \text{ for some $i$}\}\union \{y_\ell\:\; \deg_{y_\ell}(v_i)=1 \text{ for some $i$}\},
\]
and
\[
W_u=\Union_{x_\ell\notin \supp(u)}\{x_\ell^2y_\ell, x_\ell y_\ell^2\}.
\]
Note that  $u+I^k$ is annihilated by $a,b,c,d$ and all variables in $Z_u$ and all monomials in $W_u$. Indeed, it is obvious that
$a,b,c,d$ and all monomials in $W_u$, annihilate $u+I^k$. Now let $x_\ell\in Z_u$, we show that $ux_\ell\in I^k$. We can assume that $v_1=a^4x_\ell y_\ell^2$. Hence $a^6(b^4x_\ell^2y_\ell)v_2\cdots v_{k-1}\in I^k$ and  $a^6(b^4x_\ell^2y_\ell)v_2\cdots v_{k-1}|ux_\ell$, so $ux_\ell\in I^k$. Similarly for  $y_s\in Z_u$, we show that $uy_s\in I^k$.

It follows from this observation  that  $(I^k\:(c,d))/I^k$ is generated as $K$-module by the residue classes of  monomials $uvw$ where $u\in \mathcal{S}_k$,  $v$ is a monomial in the variables  $x_i$ and $y_j$ belonging to $V_u=\supp(u)\setminus Z_u$ and $w$ is a monomial in the variables not belonging to the  support of $u$ and  not divisible by a monomial in $W_u$.

\medskip
Fix $u=a^4b^4v_1\cdots v_{k-1}\in\mathcal{S}_k$ and let $m=uvw$ be a generator of $(I^k\:(c,d))/I^k$ as described in the preceding paragraph.  Then  $v$ is a monomial with $\deg_{x_i}(u)=\deg_{y_j}(u)=2$ for each $x_i,y_j\in\supp(v)$. After relabeling of the variables we may assume that
\[
\supp(u)=\{a,b, x_1,y_1,\ldots,x_t,y_t\}.
\]
Then
\begin{eqnarray}
\label{weakpoint}
uv=a^4b^4\prod_{i=1}^r(a^4x_iy_i^2)(b^4x_i^2y_i)\prod_{j=r+1}^sa^4x_jy_j^{h_{j}}\prod_{\ell=s+1}^tb^4x_\ell^{g_\ell}y_\ell
\end{eqnarray}
with $h_j\geq 2$ and $g_\ell\geq 2$, and $k-1=r+t$.

\medskip
\noindent
Claim $(*)$:  None of the monomials $m=uvw$  belongs to $I^k$.

\medskip
For the proof of claim $(*)$ we first observe

\medskip
\noindent
$(\sharp)$ If  $w_1\cdots w_s$ divides $m$ with  $w_1,\ldots,  w_s\in G(L)$, then $s\leq k-1$ and after renumbering of the  $v_i$ we have $w_i=v_i$ for $i=1,\ldots,s$.

\medskip
Indeed we may assume that  $w_1=a^4x_jy_j^2 $.   It follows from (\ref{weakpoint}) that
$x_jy_j^2$  appears in one of $v_i$. Hence after renumbering we may assume that  $w_1=v_1$. Then $w_2\cdots w_s$ divides $m/v_1$. Induction on $k$ completes  the proof of $(\sharp)$.

\medskip
Now in order to prove $(*)$ we assume in the contrary  that $m\in I^k$. Then  there exist $w_i\in G(I)$  such that $w_1\cdots w_k$ divides $m$. We may assume that $w_1,\ldots,w_s\in G(L)$ and $w_{s+1},\ldots,w_k\in G(J)$. By $(\sharp)$ we may assume that  $w_i=v_i$ for $i=1,\ldots,s$. Our next claim is the following:

\medskip
\noindent
$(\sharp\sharp)$ $s=k-1$.

\medskip
For the proof of $(\sharp\sharp)$  we consider the following two cases:

\medskip
(i) Assume $s=k-2$.  Therefore,  $w_{k-1}w_k$ divides $a^4b^4v_{k-1}vw$. However, since
$w_{k-1}w_k\in G(J)$, we have  $\deg_a(w_{k-1}w_k)>\deg_a(a^4b^4v_{k-1}vw)$ or $\deg_b(w_{k-1}w_k)>\deg_b(a^4b^4v_{k-1}vw)$, a contradiction.

\medskip
(ii) Assume $s=k-h$ with $h>2$.
Hence $w_{k-h+1}\cdots w_k$ divides $a^4b^4v_{k-h+1}\cdots v_{k-1}vw$. Since
$w_{k-h+1},\ldots,w_k\in G(J)$, it follows that
\[
\deg_a(w_{k-h+1}\cdots w_k)+\deg_b(w_{k-h+1}\cdots w_k)\geq 6h.
\]
On the other hand, $$\deg_a(a^4b^4v_{k-h+1}\cdots v_{k-1}vw)+\deg_b(a^4b^4v_{k-h+1}\cdots v_{k-1}vw)=4h+4.$$ Now since $h>2$, it follows that $4h+4<6h$. This means that
\begin{eqnarray*}
\deg_a(a^4b^4v_{k-h+1}\cdots v_{k-1}v)&+&\deg_b(a^4b^4v_{k-h+1}\cdots v_{k-1}v)\\
&<&\deg_a(w_{k-h+1}\cdots w_k)+\deg_b(w_{k-h+1}\cdots w_k),
\end{eqnarray*}
a contradiction. This concludes the proof  $(\sharp\sharp)$.

\medskip
Now as we know that   $s=k-1$, it follows that   $w_k$ divides $a^4b^4w$. This is a contradiction, since $w_k\in G(J)$. Thus the proof of  $(*)$ is completed.

\bigskip
From claim $(*)$ it  follow that $\depth(S/I^k)>0$ for even $k$. Indeed suppose that $\depth(S/I^k)=0$.  Then $I^k: \mm\neq I^k$. Since $I^k:\mm\subset I^k:(c,d)$, it follows that there exists a monomial $m=uvw\in I^k:\mm$ of the form as described before. Now since $k$ is even and $k\leq 2n$ and $v_i\neq v_j$ for $i\neq j$, the set $V_u\neq \emptyset$.  It follows that $mv'\not \in I^k$ for any $v'\in V_u$, a contradiction.

\medskip
In the next step we show that $\depth(S/I^k)\leq 1$ (and hence $\depth(S/I^k)= 1$) for even $k$ with $k\leq 2n$. Indeed, we claim that $P=(a,b,c,d,x_1,y_1,\ldots,x_{n-1},y_{n-1},x_n)$ belongs to $\Ass(I^k)$   for even $k$ with $k\leq 2n$. Then, since $$\depth(S/I^k)\leq \min\{\dim(S/Q)\:\; Q\in \Ass(I^k)\}$$ (see \cite[Proposition 1.2.13]{BH}), the required inequality follows.

To show this we note that $P\in\Ass(I^k)$ if and only if $\depth(S(P)/I(P)^k)=0$, see for example \cite[Lemma 2.3]{HRV}. Here $S(P)$ is the polynomial ring in the variables which generate $P$ and  $I(P)$ is obtained from $I$ by the substitution $y_n\mapsto 1$.

In our case  $I(P)$ is generated by
\[
a^6,a^5b,ab^5,b^6,a^4b^4c,a^4b^4d,a^4x_1y_1^2,b^4x_1^2y_1,\ldots,a^4x_{n-1}y_{n-1}^2,b^4x_{n-1}^2y_{n-1}, a^4x_{n},b^4x_{n}^2.
\]
We claim that for $k=2t$ with $t\leq n$ the monomial
\[
u'=a^{8}b^{4}(a^4x_1y_1^2)(b^4x_1^2y_1)\cdots(a^4x_{t-1}y_{t-1}^2)(b^4x_{t-1}^2y_{t-1})x_ty_t\cdots x_{n-1}y_{n-1}x_n
\]
satisfies $u'\in (I(P)^k\:\mm(P))\setminus I(P)^k$.  This shows that $\depth(S(P)/I(P)^k)=0$. Let
\[
v'_i=(a^4x_n)v_i \text{ for } i=1,\ldots,2n+2 \text{ and } v'_{2n+3}=a^6(b^4x_n^2)\prod_{i=1}^{t-1}(a^4x_iy_i^2)(b^4x_i^2y_i)
\]
where $v_i$  is defined as in the first part of the proof.  Clearly $v'_i\in I(P)^k$ for $1\leq i\leq 2n+3$. one easily see that
\[
v'_1\,|\,a\cdot u', \;\;\; v'_2\,|\,b\cdot u', \;\;\; v'_3\,|\,c\cdot u', \;\;\; v'_4\,|\,d\cdot u'.
\]
Moreover

\[
v'_{2\ell+3}\,|\,x_{\ell}\cdot u', \;\;\; v'_{2\ell+4}\,|\,y_{\ell}\cdot u, \;\;\; 1\leq\ell\leq n-1\text{ and } v'_{2n+3}\,|\,x_n\cdot u'.
\]
Hence $u'\in (I(P)^k\:\mm(P))$.

 With the same argument as given in the first part of the proof, one can easily see that $u'\not\in I(P)^k$. Therefore $u'\in (I(P)^k\:\mm(P))\setminus I(P)^k$, so $\depth(S(P)/I(P)^k)=0$, as desired.

\medskip
Finally we show that $\depth(S/I^k)=2$ for $k> 2n+1$. Since the only generators of $I^k$ which are divisible by $c$ are among the generators of $JL^{k-1}$ we see that $I^k\:(c)/I^k$ is generated by the set of monomials $\Union_{u\in\mathcal{S}_k}\{u,ud\}$. Since $k>2n+1$, it follows that $\mathcal{S}_k=\emptyset$. Hence $I^k\:(c)=I^k$ for $k>2n+1$. Similarly, $I^k\:(d)=I^k$ for $k>2n+1$. It follows that $c,d$ is a regular sequence on $S/I^k$ for $k>2n+1$. This implies that $\depth(S/I^k)\geq 2$ for all $k>2n+1$.

Let $\bar{S}=K[a,b,x_1,y_1,\ldots,x_n,y_n]$ and $$\bar{I}=(a^6,a^5b,ab^5,b^6,a^4x_1y_1^2,b^4x_1^2 y_1,\ldots ,a^4x_ny_n^2,b^4x_n^2 y_n)\subset \bar{S}.$$ Then $(S/I^k)/(c,d)(S/I^k)=\bar{S}/\bar{I}^k.$

We claim  that $w=a^5b^{6k-6}x_1y_1x_2y_2\cdots x_ny_n\in \bar{I}^k:\nn\setminus \bar{I}^k$ for $k\geq 2$, where $\nn$ is the graded maximal ideal of $\bar{S}$. The claim implies that  $\depth(S/I^k/(c,d)S/I^k)=0$ for all $k\geq 2$. In particular it follows that  $\depth(S/I^k)= 2$ for all $k>2n+1$, as desired.

To prove the claim we notice that $aw$ is divisible by $(a^6)(b^6)^{k-1}\in \bar{I}^k$, and $bw$ is divisible by $(a^5b)(b^6)^{k-1}\in \bar{I}^k$. Hence $aw,bw\in  \bar{I}^k$.

Next observe that $x_iw$ is divisible by $(a^5b)(b^6)^{k-2}(b^4x_i^2y_i)\in \bar{I}^k$ and $y_iw$ is divisible by $(b^6)^{k-1}(a^4x_iy_i^2)\in \bar{I}^k$. This implies that $x_iw,y_iw\in \bar{I}^k$ for all $i$. Thus we have shown that $w\in \bar{I}^k:\nn$.

It remains to be shown that $w\not \in \bar{I}^k$. Indeed, none of the generators of $L$ divides $w$, because each of these generators has $x_i$-degree  or $y_i$-degree $2$. Therefore, if $w\in\bar{I}^k$, it follows that $w$ is divisible by a monomial in $a$ and $b$  of degree $6k$. However, $a^5b^{6k-6}$ has only degree $6k-1$, a contradiction.

\end{proof}

\end{document}